\documentclass[12pt,oneside]{amsart}

\usepackage{amssymb,amsmath}
\usepackage{hyperref}
\usepackage{amsthm,amsfonts}
\usepackage{hyperref}
\usepackage{footnote}
\usepackage{textcomp}
\usepackage[english]{babel}
\usepackage{graphicx, color}
\usepackage{xargs}                     
\usepackage[pdftex,dvipsnames]{xcolor}  
\usepackage[colorinlistoftodos,prependcaption,textsize=tiny]{todonotes}
\newcommandx{\unsure}[2][1=]{\todo[linecolor=red,backgroundcolor=red!25,bordercolor=red,#1]{#2}}
\newcommandx{\change}[2][1=]{\todo[linecolor=blue,backgroundcolor=blue!25,bordercolor=blue,#1]{#2}}
\newcommandx{\info}[2][1=]{\todo[linecolor=OliveGreen,backgroundcolor=OliveGreen!25,bordercolor=OliveGreen,#1]{#2}}
\newcommandx{\improvement}[2][1=]{\todo[linecolor=Plum,backgroundcolor=Plum!25,bordercolor=Plum,#1]{#2}}
\newcommandx{\thiswillnotshow}[2][1=]{\todo[disable,#1]{#2}}

\numberwithin{equation}{section}
\newtheorem{theorem}{Theorem}[section]
\newtheorem{lemma}[theorem]{Lemma}
\newtheorem{proposition}[theorem]{Proposition}
\newtheorem{corollary}[theorem]{Corollary}
\newtheorem{conjecture}[theorem]{Conjecture}

\newtheorem*{theoremX}{Theorem}

\theoremstyle{definition}
\newtheorem{definition}[theorem]{Definition}

\newtheorem{remark}[theorem]{Remark}

\numberwithin{equation}{section}

\newcommand{\reg}{\operatorname{reg} }
\newcommand{\areg}{\operatorname{areg} }

\newcommand{\gens}{\operatorname{mingens} }
\newcommand{\sd}{\operatorname{sdefect} }
\newcommand{\Ass}{\operatorname{Ass} }

\newcommand{\hilb}{\operatorname{HF} } 
\newcommand{\PP} {{\mathbb{P}}}
\newcommand{\NN} {{\mathbb{N}}}

\newcommand{\K} {{\mathbb{K}}}
\usepackage{relsize}
\begin{document}
\sloppy

\title[rational normal curves]{regularity and symbolic defect of points on rational normal curves}
\thanks{}

\author{Iman Bahmani Jafarloo}
\address{DISMA, Dipartimento di Scienze Matematiche, Politecnico di Torino, Corso Duca degli Abruzzi 24, 10129 Turin, Italy. \newline \hspace*{0.3cm}  Dipartimento di Matematica, Universit\`a degli Studi di Torino, Via Carlo Alberto 10, 10123 Turin, Italy.}
\email{iman.bahmanijafarloo@polito.it}

\author{Grzegorz Malara}
\address{Pedagogical University of Cracow, Institute of
	Mathematics, Podchor\c{a}\.{z}ych 2, PL-30-084 Kraków, Poland}
\email{grzegorzmalara@gmail.com}

\keywords{}
\subjclass[2020]{}

\begin{abstract}
In this paper we study ideals of points lying on rational normal curves defined in projective plane and projective $ 3 $-space. We give an explicit formula for the value of Castelnuovo-Mumford regularity for their ordinary powers. Moreover, we compare the $ m $-th symbolic and ordinary powers for such ideals in order to show whenever the $ m $-th symbolic defect is non-zero.

\end{abstract}

\maketitle

\section{Introduction}
Studying Castelnuovo-Mumford regularity $\reg(I)$ of a homogeneous ideal $I \subseteq \K[x_0,\ldots,x_n]$ has a long story starting from the paper of Mumford \cite{Mum66}, who introduced the concept of $m$-regularity for an ideal $I$, i.e. the number $m$ for which all $i$-th syzygies of $I$ are generated in degrees not greater than $m+i$, for all $i$. Bayer and Stillman in \cite{BS}went on with Mumford's ideas by showing an explicit criterion for $m$-regularity. They also proved an equality between $\reg(I)$ and the regularity of initial ideal of $I$ with respect to the reverse lexicographic order in any characteristic of $\K$. A connection between Castelnuovo-Mumford regularity and syzygies of given ideal $I$ justifies why $\reg(I)$ can be viewed as a measure of complexity of $I$ and also explain unflagging interests in this subject.

Swanson in \cite{Swa} analyses $r$-th ordinary powers $I^r$ of homogeneous ideals $I$, showing that these powers can be expressed in terms of primary decomposition of $I^r$. As an additional result, it has been proved that $\reg(I^r)$ is bounded above by some linear functions which depend on $r$. As a consequence, a new way of investigation of $\reg(I^r)$ has begun. In \cite{CHT99} Cutkosky, Herzog, and Trung, building upon papers of Swanson and the paper of Bertram, Ein and Lazarsfeld \cite{BEL}, introduced a new asymptotic invariant, the so-called asymptotic regularity $\areg(I)$ of a homogeneous ideal $I$. Later, the work on regularity of homogeneous ideals and their powers was significantly improved in \cite{Char04,Eisenbud2004TheRO}, for instance for the case of Gorenstein and zero dimensional ideals.

One of the best known classes of curves in projective spaces $ \PP^n $ are rational normal curves and they have been studied widely, see \cite{carlini2007existence,CG1996,CEG99,conca2000}. Studying schemes of fat points lying on rational normal curves has its own long history. In \cite{CG1996} Catalisano and Gimigliano  gave an algorithm for computing the Hilbert function for fat point schemes lying on a twisted cubic curve and they extended the work for rational normal curves in $ \PP^n $ together with Ellia \cite{CEG99}. At the same time, Conca in \cite{conca2000} described the Hilbert function and resolution of symbolic and ordinary powers of ideals of rational normal curves.
 
Our motivation for this work is computing the regularity of powers of ideals of points on two types of rational normal curves, conic and twisted cubic curve. The main results of this paper concerning the regularity of powers of such ideals are Theorems \ref{th:regConic} and \ref{th:regTCC} which can be summed up as follows:
\begin{theoremX}
	Let $n\in \{2,3\}$ and let $ C\subset \PP^n $ be a rational normal curve. Denote by $I_{D_j}$ the ideal defining a set of $s$ general points on $ C $. Let $0\leq j < n$ be such that $s=nd-j$, then
	$$ \reg(I_{D_j}^r)=\begin{cases}
	rd+1  \quad \mbox{if}\   j=0\\
	rd  \ \ \qquad \mbox{if}\ 0<j<n. 
	\end{cases}
	$$
\end{theoremX}

The paper is organized as follows. In Section \ref{sec:prelim} we recall all needed definitions and prove basic facts that are used through the paper. The first non-trivial case of a rational normal curve is a conic in $\PP^2$. We dedicate Section \ref{sec:ptsOnConic} to this case. It culminates with the proof of Theorem \ref{th:regConic}. Section \ref{sec:regTCC} is devoted to the study of twisted cubic curves and the culmination of this section is Theorem \ref{th:regTCC}. The last section is a small step towards understanding the structure of symbolic powers of ideals $I_{D_j}$. We prove that for all integers $m\geq 3$ there is $I_{D_j}^{(m)} \not\subseteq I_{D_j}^m$, and state a conjecture about the relation between symbolic and ordinary powers of ideals $ I_{D_j} $. 

\paragraph*{\emph{\textbf{Acknowledgement.}}} We want to warmly thank Enrico Carlini and Tomasz Szemberg for all inspiring discussions through the whole process of writing this article. Bahmani Jafarloo thanks Navid Nemati and Aldo Conca for useful conversations. 

Bahmani Jafarloo was partially supported by MIUR grant Dipartimento di Eccellenza 2018-2022 (E11G18000350001), Malara was partially supported by National Science Centre, Poland, Sonata Grant 2018/31/D/ST1/00177.

\section{Preliminary}
\label{sec:prelim}
Let $ S= K[x_0,\ldots,x_n] $ be the graded ring of polynomials over an algebraically closed filed $K$. Let
$$M=
\begin{pmatrix}
x_0 &x_1  & \cdots & x_{n-1}\\
x_1 &x_2 &\cdots &x_{n}
\end{pmatrix}.
$$
Denote by $ I=I_2(M) $ be the ideal generated by the $2$-minors of $ M $ (known as the Hankel matrix). It is known that the ideal $ I $ defines the rational normal curve (RNC for short) in $ \mathbb P^n $, which we denoted by $ C $, the Veronese embedding of 
\begin{equation}\label{verembed}
\nu_n:\mathbb{P}^1\hookrightarrow \mathbb{P}^n,\quad[s:t]\mapsto[s^n:s^{n-1}t:\cdots: st^{n-1}:t^{n} ].
\end{equation}

Recall that for any homogeneous ideal $J$ the Hilbert function $ \hilb(S/J,t) $ of $ S/J $, for $ t\in\mathbb{N}\cup\{0\} $, is the dimension over $K$ of degree $t$ homogeneous part of $S/J$. 
\begin{remark}
For the ideal $I=I_2(M)$ the Hilbert function of $S/I$ is known to be  $$\hilb(S/I,t)=n(t+1)-(n-1), \quad \text{for}\ t \geq 0.$$
\end{remark}
Let $ J\subset S$ be any homogeneous ideal. We denote by $ \beta_{ij}(J) $ the $ (i,j) $-th \emph{Betti number} of $ J $, i.e. the dimension of $ \operatorname{Tor}_i^S(J,K) $ in degree $ j $. By definition, the Castelnuovo-Mumford regularity $\reg(J)$ of $J$ is $$\operatorname{reg}(J)=\max \left\{j-i: \beta_{i j}(J) \neq 0\right\}.$$
It is convenient to write $ \beta(J) $ and $ \alpha(J) $ for the maximum and the minimum degree of the minimal set of generators of $ J $, respectively. In general we have that  $ \reg(J)\geq \beta(J) $ and $ \reg(S/J)=\reg(J)-1 $.
\begin{remark}
 It is known that for a zero-dimensional ideal $ J $, if $ t\geq 0 $ is the least value such that $\Delta\hilb(S/J,t)=\hilb(S/J,t)-\hilb(S/J,t-1)=0$, then $ \reg(J)=t $. 
\end{remark}
\begin{definition}
	\label{def:aReg}
	Let $J\subset S$ be a homogeneous ideal. Then the asymptotic regularity of $J$ is the real number
	$$\areg(J) =\lim_{r \rightarrow \infty} \frac{\reg(J^r)}{r}.$$
\end{definition}
At it was shown in \cite[Theorem 1.1]{CHT99}, we have always that $\areg(J)= \frac{\beta(J^r)}{r}$, since it is known that $\beta(J^r)$ is linear function which depends on $r$ for all $r\gg 0$.

Let $ D_j\subset C $ be a set of $ nd-j $ general points on the rational normal curve $ C\subset \PP^n $ for integers $ d\geq 2 $ and $0\leq j\leq n-1 $. Denote by $ I_{D_j}  $ the ideal defining the set $ D_j $. In the following we study the ideal $ I_{D_j}  $ and the next lemma is an observation that we need in order to prove that the forms of order $rd$ does not vanish in $I_{D_j}^r$.
 
 \begin{lemma}\label{maxdegIDr}
	Let $ D_j $ be a set of $ nd-j $ points on rational normal curve $ C $. Then, $\beta(I_{D_j}^r)=r\beta(I_{D_j})=rd$.
\end{lemma}
\begin{proof}
	The proof directly follows from \cite[Exercise A2.21, d]{Eis95}. More precisely, $ I_{D_j} $ is an ideal in the symmetric algebra $ S/I $ (the coordinate ring of $ C $) generated at most in degree $ d $. 
\end{proof}
\begin{proposition}\label{boundforreg}
	Let $ D_j $ be as in Lemma \ref{maxdegIDr}. If $ r\geq2  $ and $ d\geq2  $, then $$ rd\leq\reg(I_{D_j}^r)\leq \reg(I_{D_j})+(r-1)d.$$
\end{proposition}

\begin{proof}
	On the one hand from Lemma \ref{maxdegIDr} and the fact that $ \beta(I_{D_j}^r)\leq \reg(I_{D_j}^r) $, we have that $ rd\leq\reg(I_{D_j}^r) $. On the other hand since $I_{D_j}$ is a zero-dimension ideal generated at most in degree $ d $, therefore from \cite[Corollary 7.9]{Eisenbud2004TheRO} we have that
	$ \reg(I_{D_j}^r)\leq \reg(I_{D_j})+(r-1)d$. Hence, $$ rd\leq\reg(I_{D_j}^r)\leq \reg(I_{D_j})+(r-1)d.\qedhere$$
\end{proof}

\begin{lemma}\label{distinctpoints}
	The set $\{x_0^{d-1}-x_n^{d-1}=0\} $ and $ C $ meet each other exactly at $n(d-1)  $ distinct points.
\end{lemma}
\begin{proof}
One can see that
	$$ x_0^{d-1}-x_n^{d-1} =\prod_{i=1}^{d-1}(x_0-\xi_ix_n),$$
	where $\xi_i$ is the $i$-th primitive root of unity for $ i=1,\ldots,d-1 $. By (\ref{verembed}) we have that
	\begin{flalign*}
	\prod_{i=1}^{d-1}(x_0-\xi_ix_n)= \prod_{i=1}^{d-1} (s^n-\xi_it^n)= \prod_{i=1}^{d-1}(\zeta^n-\xi_i). 
	\end{flalign*}
	It follows that $\{x_0^{d-1}-x_n^{d-1}=0\} $ intersects $ C $ at $ n(d-1) $ distinct points, therefore the desired result follows. Moreover, we conclude that no two hyperplanes $\{x_0-\xi_\alpha x_n=0\}$ and $\{x_0-\xi_\beta x_n=0\}$, with $\alpha\neq \beta$, intersect $C$ at the same point for all $\alpha,\beta\in\{1,2,\ldots,d-1\} $. 
\end{proof}

In the following sections, we study the regularity of $ I_{D_j}^r $ where $ D_j $ lies on a conic in $ \PP^2 $, or on a twisted cubic curve (TCC) in $ \PP^3 $. Since we are considering these points in two separate sections, we agree to use the same notation of $C$ for both, conic and TCC.
\section{Regularity of points on a conic}
\label{sec:ptsOnConic}
This section is devoted to study the regularity of $ I^r_{D_j} $ where $ D_j\subset C \subset\PP^2$.
By the definition of ideal $I$, we have that
$$I=\det
\begin{pmatrix}
x_0 &x_1 \\
x_1 &x_2
\end{pmatrix} =\left\langle x_1^2-x_0x_2\right\rangle= \left\langle Q\right\rangle.
$$
\begin{lemma}\label{distpointsn2}
Let $ D_j $ be a set of $ 2d-j $ distinct points in $ \PP^2 $ lie on $ C $ for $ d\geq 2 $ and $ j\in\{0,1\} $. Then its defining ideal can be represented as:
$$I_{D_j}=\begin{cases}
I+\left\langle x_1(x_0^{d-1}-x_2^{d-1})\right\rangle,   \qquad\qquad\qquad \qquad\qquad \mbox{if}\   j=0\\
I+\left\langle x_1(x_0^{d-1}-x_2^{d-1}),x_0(x_0^{d-1}-x_2^{d-1})\right\rangle,   \ \ \qquad \mbox{if}\ j=1.
\end{cases}
$$
\end{lemma}
\begin{proof}
		We proceed as follows:
\begin{itemize} 
	\item Let $ j=0 $. By Lemma \ref{distinctpoints} one can see that $\{x_1=0\} \cap \{x_0 - \xi_{\alpha}x_2=0\} \cap C=\emptyset$ for $\alpha={1,2,\ldots,d-1} $. Since the line $ \{x_1=0\} $ does not contain any tangent line to $ C $, therefore the intersection of $ \{x_1(x_0^{d-1}-x_2^{d-1})=0\} $ and $ C $ is a set $ 2(d-1)+2=2d $ distinct points.
	\item Let $ j=1 $. Since the point $\{ \langle x_1,x_0\rangle\}\not\in \{x_0^{d-1}-x_2^{d-1}=0\} $, the desired result follows from Lemma \ref{distinctpoints}.
\end{itemize}

\end{proof}

 \begin{proposition}\label{propreg2}
	Let $ D_j $ be as in Lemma \ref{distpointsn2}. Then,
	$$ \reg(I_{D_j})=\begin{cases}
	d+1,  \quad \mbox{if}\   j=0,\\
	d,  \ \ \qquad \mbox{if}\ j=1. 
	\end{cases}
	$$
\end{proposition}
\begin{proof}
Let $ j=0 $. Then the syzygy matrices of $ S/I_{D_0} $ are as follows
$$A_1=\begin{pmatrix}
Q & x_1(x_0^{d-1}-x_2^{d-1})
\end{pmatrix},\ A_2=\begin{pmatrix}
 x_1(x_0^{d-1}-x_2^{d-1})\\
 	-Q
\end{pmatrix}.$$
Therefore, we have its minimal free resolution 
$$0\xrightarrow{}S(-d-2)\xrightarrow{A_2}S(-2)\oplus S(-d)\xrightarrow{A_1}S\xrightarrow{}S / I_{D_0}\xrightarrow{}0,$$
and from that $ \reg( S/I_{D_0})=d $. Accordingly, $ \reg(I_{D_0})=d+1 $ 

Similarly, for $ j=1 $, we compute the syzygy matrices for $ S/I_{D_1} $,
$$A_1=\begin{pmatrix}
Q & x_1(x_0^{d-1}-x_2^{d-1}) & x_0(x_0^{d-1}-x_2^{d-1})
\end{pmatrix},\ A_2=\begin{pmatrix}
0 & x_0^{d-1}-x_2^{d-1}\\
x_0 & -x_1\\
-x_1 &x_2
\end{pmatrix}.$$
Hence,
$$0\xrightarrow{}S^2(-d-1)\xrightarrow{A_2}S(-2)\oplus S^2(-d)\xrightarrow{A_1}S\xrightarrow{}S / I_{D_1}\xrightarrow{}0.$$
We see that $ \reg( S/I_{D_1} ) = d-1 $, consequently $ \reg(I_{D_1} ) = d $.
\end{proof}
\begin{lemma}\label{uperd0n2}
Let $ D_0 $ be as in Lemma \ref{distpointsn2}. Then, $ \reg(I_{D_0}^r)\geq rd+1 $ for $r\geq 2  $.
\end{lemma}
\begin{proof}
Set $ G= x_1(x_0^{d-1}-x_2^{d-1}) $. Directly from the definition of ordinary power $  I_{D_0}^r =\left\langle \left\lbrace Q^{r-t}G^t\right\rbrace_{t=0}^r  \right\rangle $. Hence, the first syzygy matrix of  $ S/ I_{D_0}^r $ is
$$A_1=\begin{bmatrix}
Q^r & Q^{r-1}G & Q^{r-2}G^2 & \cdots  & Q^2G^{r-2} & QG^{r-1} & G^{r}
\end{bmatrix}.$$
It is a straightforward computation that the second syzygy matrix can be express in the following manner
$$
A_2=\underbrace{\begin{bmatrix}
-G & 0 &0 & 0 & \cdots  \\
Q & -G &0 & 0 & \cdots  \\
0 & Q &-G & 0  & \cdots \\
0 & 0 &Q & -G & \cdots  \\
0 & 0 &0 & Q  & \cdots \\
0 & 0 &0 & 0 & \cdots \\
\vdots  & \vdots  &\vdots & \vdots & \cdots \\
0 & 0 &0 & 0 & \cdots \\
0 & 0 &0 & 0 & \cdots \\
 0 & 0 & 0 & 0 & \cdots \\
	\end{bmatrix}}^{A_{21}}_{S(-(2r+1))\cdots}
\cdots
\underbrace{\begin{bmatrix}
 0\\
0\\
0\\
0\\
0 \\
0\\
\vdots\\
-G\\
Q\\
0\\
	\end{bmatrix}}^{A_{22}}_{S(-(rd+1))}
\cdots
\underbrace{\begin{bmatrix}
 0\\
0\\
0\\
0\\
0 \\
0\\
\vdots\\
0\\
-G\\
Q \\
\end{bmatrix}}^{A_{23}}_{S(-(rd+2))}.
$$
This proves that $ \reg(S/I_{D_0}^r)\geq rd $, in consequence $ \reg(I_{D_0}^r)\geq rd+1 $.
\end{proof}
\begin{theorem}
	\label{th:regConic}
Let $ D_j $ be as in Lemma \ref{distpointsn2}. If $ r\geq2  $, then
\begin{itemize}
	\item[(1)] $ \reg(I_{D_0}^r)=rd+1$,
	\item[(2)] $ \reg(I_{D_1}^r)=rd$.
\end{itemize}
\end{theorem}
\begin{proof}
(1) follows from Propositions \ref{boundforreg},\ref{propreg2} and Lemma \ref{uperd0n2}. 
The proof of (2) follows directly from Propositions \ref{boundforreg},\ref{propreg2}.
\end{proof}

\section{Regularity of points on a TCC}
\label{sec:regTCC}
Let $ n=3 $. In this section we study the $ \reg(I_{D_j}^r) $, where $ D_j $ is a set of points $ 3d-j $ lie on the twisted cubic curve $ C $ defined by the following ideal,  
$$I=\left\langle x_2^2-x_1x_3,x_1x_2-x_0x_3,x_1^2-x_0x_2\right\rangle= \left\langle Q_1,Q_2,Q_3\right\rangle.$$
 \begin{lemma}\label{distincpoint3}
 The ideal $ I_{D_j} $ defines the set  $ D_j $ of  $3d-j$ distinct points on $ C $ for $ j=0,1,2 $.
 	 $$ I_{D_j}=\begin{cases}
 	I+\left\langle (x_2-x_1)(x_0^{d-1}-x_3^{d-1})\right\rangle,    \quad \mbox{if}\   j=0\\
 	I+\left\langle x_2(x_0^{d-1}-x_3^{d-1}),x_1(x_0^{d-1}-x_3^{d-1})\right\rangle,   \ \ \qquad \mbox{if}\ j=1\\
 	I+\left\langle x_2(x_0^{d-1}-x_3^{d-1}),x_1(x_0^{d-1}-x_3^{d-1}),x_0(x_0^{d-1}-x_3^{d-1})\right\rangle, \quad \mbox{if}\   j=2.
 	\end{cases}
 	$$
 \end{lemma}
 \begin{proof}
 	We divide the proof into three cases as the following.
 	\begin{itemize} 
 		\item Let $ j=0 $. By Lemma \ref{distinctpoints}, an elementary calculation shows that planes $\{x_2-x_1=0\}$ and $\{x_0 - \xi_{\alpha}x_3=0\}$ do not meet $C$ at the same point for $\alpha={1,2,\ldots,d-1} $. Since the plane $ \{x_2-x_1=0\} $ does not contain any tangent line to $ C $, we conclude that $  \{(x_2-x_1)(x_0^{d-1}-x_3^{d-1})=0 \}$ intersects $ C $ at $ 3(d-1)+3=3d $ points.
 		
 		\item Let $ j=1 $. we have that  $$\langle x_2(x_0^{d-1}-x_3^.{d-1}),x_1(x_0^{d-1}-x_3^{d-1})\rangle=\langle x_2,x_1\rangle\langle x_0^{d-1}-x_3^{d-1}\rangle.$$
 		One can see that the line $ \{\left\langle x_2,x_1\right\rangle\} $ is not tangent to $ C $, moreover by Lemma \ref{distinctpoints}, we have that $\{\left\langle x_2,x_1\right\rangle\}\cap \{x_0^{d-1}-x_3^{d-1}=0\}\cap C =\emptyset $. Therefore, $ \{\langle x_2,x_1\rangle\langle x_0^{d-1}-x_3^{d-1}\rangle \} $ intersects $ C $ at $ 3(d-1)+2 =3d-1$. 
 		
 		\item Let $ j=2 $. Since the point $\langle x_2,x_1,x_0\rangle\not\in  \{x_0^{d-1}-x_3^{d-1}=0\} $, therefore by Lemma \ref{distinctpoints} the desired result follows.
 	\end{itemize}
 	This completes the proof.\qedhere
 \end{proof}
 \begin{proposition}\label{propreg}
 	Let $ D_j $ be as in Lemma \ref{distincpoint3}. Then,
 	$$ \reg(I_{D_j})=\begin{cases}
 	d+1,  \quad \mbox{if}\   j=0,1\\
 	d,  \ \ \qquad \mbox{if}\ j=2. 
 	\end{cases}
 	$$
 \end{proposition}
 \begin{proof}
 	We are looking for minimal free resolutions of the form
 	$$0\xrightarrow{} F_3\xrightarrow{A_3}F_2\xrightarrow{A_2}F_1\xrightarrow{A_1} F_0\xrightarrow{}R / I_{D_j}\xrightarrow{}0$$
 	for any ideal $I_{D_j}$. Since for any $j$ we know the generators of ideals $I_{D_j}$, we can write matrices $A_i$ explicitly. 
 	
 	For the sake of the completeness, denote by $H=x_0^{d-1}-x_3^{d-1}$. With some aids of any algebraic software program, such as \verb|Macaulay2| \cite{MC2}, we compute the syzygy matrices of $S/I_{D_j}.$
 	In case of $j=0$ we have
 		\[
 		A_1= \left[\begin{array}{cccc}
 		Q_1 &		Q_2 & 		Q_3 & 		(x_2-x_1)H
 		\end{array} \right],
 		\]
 		\begin{small}
 		$$
 		A_2= \left[ \begin{array}{ccccc}
 		x_1 & x_0 & x_2x_3^{d-1}  & -x_2x_3^{d-1} & 0 \\
 		-x_2&-x_1&x_2(H-x_0^{d-2}x_1)& x_1(H-x_0^{d-2}x_2)&0\\
 		x_3 & x_2 &  x_0^{d-2}x_2^2 - x_3H & x_0^{d-2}x_2^2 - x_3H & (x_1-x_2)H \\
 		0 & 0 & -Q_1 & -Q_1-Q_2 & -Q_1
 		\end{array} \right],
 		$$
 	    \end{small}
 		\[ 
 		A_3= \left[ \begin{array}{cc}
		x_2H&x_0H\\
		-x_0^{d-2} x_1 x_2&      -x_1H-x_2 x_3^{d-1}\\
		x_1+x_2&              x_0+x_1\\
		-x_2&                  -x_1\\
		x_3&                   x_2
		\end{array} \right],
		\]
	therefore the minimal free resolution is
 	$$0\xrightarrow{}S^2(-d-3)\xrightarrow{A_3}S^2(-3)\oplus S^3(-d-2)\xrightarrow{A_2}S^3(-2)\oplus S(-d)\xrightarrow{A_1}S\xrightarrow{}S / I_{D_0}\xrightarrow{}0.$$
 	While for $j=1$ there is
 	\begin{small}
 		\[
 		A_1= \left[ \begin{array}{ccccc}
 		Q_1 & 		Q_2 &  		Q_3 & 		x_2H & 		x_1H
 		\end{array} \right],
 		\]
 		\[
	 	A_2= \left[ \begin{array}{ccccc}
	 	x_1& 0&    x_0& -x_3^{d-1}&     0\\
	 	-x_2&0&    -x_1&x_0^{d-2}x_1& 0\\
	 	x_3& 0&    x_2& -x_0^{d-2}x_2&H\\
	 	0&    x_1& 0&    -x_2&       x_0\\
	 	0&    -x_2&0&    x_3&        -x_1
	 	\end{array} \right],
 		A_3= \left[ \begin{array}{c}
 		x_1H\\
 		-Q_2\\
 		-x_0^{d-2}x_1^2+x_2x_3^{d-1}\\
 		-Q_3\\
 		Q_1
 		\end{array} \right].
 		\]
 	\end{small}
 Thus
 $$0\xrightarrow{}S(-d-3)\xrightarrow{A_3}S^2(-3)\oplus S^3(-d-1)\xrightarrow{A_2}S^3(-2)\oplus S^2(-d)\xrightarrow{A_1}S\xrightarrow{}S / I_{D_1}\xrightarrow{}0.$$
 	For the last remaining case, $j=2$, the matrices are the following
 	\begin{small}
 		\[
 		A_1= \left[ \begin{array}{cccccc}
 		Q_1 & 		Q_2 & 		Q_3 & 		x_2H & 		x_1H &
 		x_0H
 		\end{array} \right],
 		\]
 		\[
 		A_2= \left[ \begin{array}{cccccccc}
 		x_1& 0&    x_0& 0&    0&    -x_3^{d-1}&     0&            0\\
 		-x_2&0&    -x_1&0&    0&    x_0^{d-2}x_1& H&0\\
 		x_3& 0&    x_2& 0&    0&    -x_0^{d-2}x_2&0&            H\\
 		0&    x_1& 0&    x_0& 0&    -x_2&       0&            0\\
 		0&    -x_2&0&    0&    x_0& x_3&        -x_2&        -x_1\\
 		0&    0&    0&    -x_2&-x_1&0&           x_3&         x_2
 		\end{array} \right],
 		A_3= \left[ \begin{array}{ccc}
 		0&    0&     H\\
 		x_0& 0&     -x_2\\
 		0&    x_3^{d-1}&-x_0^{d-2}x_1\\
 		-x_1&x_2&  0\\
 		x_2& -x_3& 0\\
 		0&    x_0&  -x_1\\
 		0&    -x_1& x_2\\
 		0&    x_2&  -x_3
 		\end{array} \right].
 		\]
 	\end{small}
 	Hence we can write,
 	$$0\xrightarrow{}S^3(-d-2)\xrightarrow{A_3}S^2(-3)\oplus S^6(-d-1)\xrightarrow{A_2}S^3(-2)\oplus S^3(-d)\xrightarrow{A_1}S\xrightarrow{}S / I_{D_2}\xrightarrow{}0,$$
 	By a straightforward calculation from the definition of regularity, we get the desired assertion.
 \end{proof}
The minimal free resolution of $I_{D_1}  $, calculated in the previous theorem, gives us immediately the following corollary.
 \begin{corollary}
 	The ideal $I_{D_1}  $ is a Gorenstein ideal.
 \end{corollary}
\begin{lemma}\label{j=0}
	Let $ D_0 $ be as in Lemma \ref{distincpoint3}. Then, $ \reg(I_{D_0}^r)\geq rd+1 $ for $r\geq 2  $.
\end{lemma}
\begin{proof}
Set $ G=(x_2-x_1)(x_0^{d-1}-x_3^{d-1}) $. The $ r $-th power of $ I_{D_0} $ is as the following
$$I^r_{D_0}=\left\langle Q_1^{r}, Q_1^{r-1}Q_2, \cdots , Q_1G^{r-1} , \cdots , Q_2G^{r-1} , \cdots , Q_3G^{r-1} , G^r\right\rangle .$$
Consider the 0-$ dimensional $ ideal $J =\left\langle Q_1,Q_2,Q_3,G^r\right\rangle  $. Since $ I_{D_0}^r\subset J $, therefore we have the following exact sequence:
$$0 \rightarrow I^r_{D_0} \rightarrow J \rightarrow \frac{J}{I^r_{D_0}} \rightarrow 0.$$
Hence we have that $$\operatorname{reg}(J) \leq \max \Big\{\operatorname{reg}\Big(\frac{J}{I^r_{D_0}} \Big), \operatorname{reg}(I^r_{D_0})\Big\}.$$
\textit{Claim}: $ \reg(J)= rd+1 $. Since $ I_C\subset J $ we have that $ [I_C]_t=[J]_t $ for $ t\leq rd-1 $, and it is known that $\hilb(S/I_C,t)=3t+1 $ for $ t\geq 0 $, therefore $ \hilb(S/J,t)=3t+1 $ for $ t\leq rd-1 $. We know that the degree of $ J $ is $ 3rd $, therefore either $ \hilb(S/J,rd)$ is $ 3rd-1$ or $3rd$. By contradiction assume that $\hilb(S/J,rd)= 3rd-1  $. Hence, the first difference of the Hilbert function of $ S/J $ is as follows, 
$$ \begin{array}{cccccccc}
1 &3&3&3&\cdots & 1 &1& 0
\end{array}.$$
So, by \cite[Proposition 5.2]{GMR83} follows  that $ V(J) $ contains a subset of $ rd+2 $ collinear points having multiplicities $ r $. It is a contradiction with the fact that  $ V(J) $ has only subsets of at most $ 2r$ collinear points. Therefore,
\begin{center}
	\begin{tabular}{c|ccccccccc}
		$ t $& 0 &1&2&3&$ \cdots $& $rd-1$ &$rd$ &$rd+1 $& $ \cdots $\\
		\hline
		$\hilb(S/J,t)$&1&4&7&10 & $ \cdots $&$3rd-2$ &$3rd$& $ 3rd $& $ \cdots $
	\end{tabular}.
\end{center}
We conclude that $ \reg(J)= rd+1 $. 

We know that  $$\hilb(S/(J/I_{D_0}^r),t)= \hilb(S/I^r_{D_0},t)-\hilb(S/J,t), \qquad \forall t\geq 0.$$ Since the set minimal generators of $ I^r_{D_0} $ has only one form of degree $\beta(I_{D_0}^r)=rd$, we conclude that $\hilb(S/I^r_{D_0},t)-\hilb(S/J,t)=c\in \mathbb{Z}^+,$ for all $ t\geq rd$. Therefore, the Hilbert function of $S/(J/I_{D_0}^r)$ is partially as follows:
	\begin{center}
	\begin{tabular}{c|ccccccccc}
		$ t $& 0 &1&2&3&$ \cdots $& $rd$ &$rd+1$ &$rd+2$ & $ \cdots $\\
		\hline
		$\hilb(S/(J/I_{D_0}^r),t)$&0&0&3&10 & $ \cdots $& $c$ &$ c $ &$c$ &$ \cdots $
	\end{tabular}.
\end{center}
This follows that $ \reg\Big(\frac{J}{I_{D_0}^r}\Big) $ is at most $ rd-1 $. From Proposition \ref{boundforreg}, we know that $ \reg(I_{D_0}^r)\geq rd $, hence, $\reg\Big(\frac{J}{I^r_{D_0}} \Big)<\operatorname{reg}(I^r_{D_0})  $. Therefore, 
$$rd+1=\operatorname{reg}(J) \leq \max \Big\{\operatorname{reg}\Big(\frac{J}{I^r_{D_0}} \Big), \operatorname{reg}(I^r_{D_0})\Big\}=\operatorname{reg}(I^r_{D_0}).$$
The proof is done.
\end{proof}

\begin{theorem}
	\label{th:regTCC}
Let $ D_j $ be as in Lemma \ref{distincpoint3}. If $ r\geq2  $ and $ d\geq2  $, then
\begin{itemize}
\item[(1)] $ \reg(I_{D_0}^r)=rd+1$,
\item[(2)] $ \reg(I_{D_1}^r)=\reg(I_{D_2}^r)=rd$,
\end{itemize}
\end{theorem}
\begin{proof}
The proof of (1) is a direct consequence of  Propositions \ref{boundforreg},\ref{propreg} and Lemma \ref{j=0}. The proof for $j=1$ follows from Propositions \ref{propreg},\ref{boundforreg} and \cite[Proposition 1.12.6]{Char04}. The last remaining case for $j=2$ similarly the result follows from Propositions \ref{boundforreg} and \ref{propreg}. The proof is complete.
\end{proof}
\begin{corollary}
	For the ideals $I_{D_j}$ defined in Lemma \ref{distincpoint3}, we have that 
	$$\areg(I_{D_j})=\lim_{r \rightarrow \infty} \frac{\reg(I_{D_j}^r)}{r}=d.$$	
\end{corollary}

\subsection*{Remarks in $ \PP^n $}
It is natural to ask about the regularity of the same type of ideals in higher projective spaces. However, simply calculations can show that the formula for $\reg(I_{D_j}^r)$, with $r>1$, is much more complicated than for cases of $ \PP^2 $ and $ \PP^3 $, and can not be easily described. Thus, we dedicate this section to be a leading step on further investigations in this subject, by proving the lemma which concerns $\reg(I_{D_j})$.
\begin{definition}
	\label{def:ID in PN}
	Let $ n\geq 4 $ and $ 
	0\leq j\leq n-1$. Let $ I_{D_j} $ be the ideal of a set $ nd-j $ points on $ C $ defined by the ideal $ I=I_2(M) $ as follows,
	$$\begin{cases}
	I_{D_0}=I+\left\langle (x_{n-1}-x_{1})(x_0^{d-1}-x_n^{d-1})\right\rangle,\\
	I_{D_1}=I+\left\langle x_{n-1}(x_0^{d-1}-x_n^{d-1}),x_{n-2}(x_0^{d-1}-x_n^{d-1})\right\rangle,\\
	I_{D_j}=I_{D_{j-1}}+\left\langle x_{n-j-1}(x_0^{d-1}-x_n^{d-1})\right\rangle, \quad \mbox{if}\   2\leq j\leq n-1.
	\end{cases}
	$$
\end{definition}
One can easily observe that the proof of the fact that ideals $I_{D_j}$ indeed describes the set of $ nd-j $ distinct points can be mimic from the proof of Lemma \ref{distincpoint3}. Also the next remark is similar to the result obtained in Proposition \ref{propreg}.
\begin{remark}
	For ideals $I_{D_j}$ defined as in Definition \ref{def:ID in PN}, one can compute the $ \reg(I_{D_j}) $ as in Proposition \ref{propreg} by writing their free resolutions or directly by computing their Hilbert functions,  
	$$ \reg(I_{D_j})=\begin{cases}
	d+1,  \quad \mbox{if}\    0\leq j<n-1  \\
	d,  \ \ \qquad \mbox{if}\ j=n-1. 
	\end{cases}
	$$
\end{remark}

\section{Symbolic defect}

\label{sec:sdefect}
Comparing symbolic and ordinary powers of ideals of points in $\PP^N$ has became very popular in recent years. There are a few different concepts which are concerning ``ideal containment problem". In this section we want to analyse one of them in the case of ideals $I_{D_j}$. Let us recall first the definition of symbolic power of ideal.

\begin{definition}
	Let $I$ be a homogeneous ideal in a polynomial ring $R$.
	For $m\geq 1$, the $m$-th symbolic power of $I$ is the ideal 
	$$ I^{(m)} = R \cap \left( \bigcap_{\mathfrak{p} \in \Ass(I)} (I^{m})_{\mathfrak{p}} \right), $$
	where the intersection is taken over all associated primes $\mathfrak{p}$ of $I$.
\end{definition}

It is known that for any $m$ the inclusion $I^m \subseteq I^{(m)}$ holds, but the reverse may fail. Therefore it is natural to ask about the number of generators in the module $I^{(m)}/I^{m}$.

\begin{definition}
We define the $m$-th symbolic defect of $I$ for any integer $m\geq 2$ to be 
$$\sd(I, m)=\text{ the number of minimal generators of } I^{(m)}/I^{m}.$$
\end{definition}

We refer to \cite{GGS19} readers interested in this subject.

Motivated by the result of relation between symbolic and ordinary powers obtained for ideal of $s$ general points on smooth conic \cite{BH10}, we take another step towards description of this behaviour for ideals of $s$ general points on a TCC, by analysing the symbolic defect of $I_{D_j}$. What we can prove for those ideals is the following:

\begin{theorem}\label{cont}
	Let $I_{D_j}$ be the ideals of points defined in Lemma \ref{distincpoint3}. Then
		\begin{itemize}
		\item[(1)] $\sd(I_{D_1}, m)>0$, if $ m\geq3 $.
		\item[(2)] $\sd(I_{D_j}, m)>0$ for $ j=0,2$.
	\end{itemize}
\end{theorem}
\begin{proof} Our proof is based on simply observation that a particular element, different for each case, belongs to $ I_{D_j}^{(m)} \setminus I_{D_j}^m $.
	
For the case 1) consider the polynomial
$$f_1=Q_1Q_3(x_0^{d-1}-x_3^{d-1}).$$
We prove by induction on $k \geq 1$ that
$$f_{1}^k \in I^{(3k)}_{D_{1}} ,\;\; Q_2f_{1}^k \in I^{(3k+1)}_{D_{1}}, \;\; Q_2Q_3f_{1}^k \in I^{(3k+2)}_{D_{1}},$$
while 
\begin{equation}
\label{eq:symb}
	f_{1}^k \not\in I^{3k}_{D_{1}} ,\;\; Q_2f_{1}^k \not\in I^{3k+1}_{D_{1}}, \;\; Q_2Q_3f_{1}^k \not\in I^{3k+2}_{D_{1}}.
\end{equation}
First observe that $f_{1} \in I^{(3)}_{D_{1}}$, which is a straightforward consequence of Zariski-Nagata theorem (see \cite[Theorem 3.14]{Eis95}). Assume for the induction hypothesis that we have $f_{1}^k \in I^{(3k)}_{D_{1}}$, for some $k >1$. Then one can easily check that there are $Q_2f_{1}^k \in I^{(3k+1)}_{D_{1}}, Q_2Q_3f_{1}^k \in I^{(3k+2)}_{D_{1}}$, once again by Zariski-Nagata theorem.	The fact that  $f_{1}^{k+1} \in I^{(3k+3)}_{D_{1}}$ follows from induction hypothesis together with
$$f_{1}^{k}f_{1} \in I^{(3k)}_{D_{1}}I^{(3)}_{D_{1}} \subset  I^{(3k+3)}_{D_{1}},$$
since symbolic powers of any homogeneous ideal $I$ form graded sequence of ideals.

Now we turn to the second part of the proof of 1). It can be checked by any symbolic algebra program, or check by hand, that $f_{1} \not\in I^{3}_{D_{1}}$. Therefore directly from the definition of ordinary power we get 
$$f_{1}^k \not\in I^{3k}_{D_{1}}.$$
Multiplying element $f_{1}^k$ by appropriate $Q_i \in I_{D_1}$ gives the desired assertion (\ref{eq:symb}).

The proof of the case 2) is identical as the case 1), if we instead of taking $f_1$ consider the polynomials
$$f_0=Q_3(x_2-x_3)(x_0^{d-1}-x_3^{d-1}), \;\; f_2=Q_3(x_0^{d-1}-x_3^{d-1}),$$
and proceed by induction on $k\geq1$ in order to show that
$$f_{0,2}^k \in I^{(2k)}_{D_{0,2}} ,\;\; Q_1f_{0,2}^k \in I^{(2k+1)}_{D_{0,2}} ,$$
and
$$f_{0,2}^k \not\in I^{2k}_{D_{0,2}}, \;\; Q_1f_{0,2}^k \not\in I^{2k+1}_{D_{0,2}}.$$
\end{proof}
\begin{remark}
	There is one missing case of $\sd(I_{D_1}, 2)$ in the statement of Theorem \ref{cont}. We expect that $\sd(I_{D_1}, 2)=0$, however we do not have a theoretical proof of this hypothesis.
\end{remark}
Motivated by numerous tests and observations that we made, we want to finish this section with a conjecture that we was not able to prove, but we believe to be true. 
\begin{conjecture}
	Let $D_j $ be a set of $ 3d-j$ general points on a TCC, where $0\leq j \leq 2$. Then  
	\begin{itemize}
		\item[1)]  $ I_{D_j}^{(m)}\subseteq I_{D_j}^r $ if and only if  $ m\geq r+1 $ for any integer $r\geq 2$, in the case  $ j=0,2$.
		\item[2)] $ I_{D_1}^{(m)}\subseteq I_{D_1}^r $ if and only if  $ m\geq r+1 $ for $ r\geq 3 $, and moreover, $ I_{D_1}^{(m)}\subseteq I_{D_1}^2 $ if and only if $ m\geq 2 $.
	\end{itemize}
\end{conjecture}

\bibliographystyle{abbrv}
\bibliography{rnc} 

\end{document}